\documentclass[oneside,english]{amsart}
\usepackage[T1]{fontenc}
\usepackage[latin9]{inputenc}
\usepackage{amsthm}
\usepackage{amstext}
\usepackage{amssymb}
\usepackage{esint}
\usepackage{graphicx}
\usepackage{enumerate}
\usepackage{amscd}

\makeatletter
\newtheorem{theorem}{Theorem}[section]
\newtheorem{lemma}[theorem]{Lemma}

\numberwithin{figure}{section}
\theoremstyle{definition}

\theoremstyle{remark}
\newtheorem{remark}[theorem]{Remark}

\numberwithin{equation}{section}
\makeatother

\usepackage{color}

\usepackage{babel}

	\DeclareMathOperator{\loc}{loc}

\newcommand\R{\mathbb R}

\begin{document}

\title[Neumann $(p,q)$-eigenvalue problem]{On $(p,q)$-eigenvalues of the weighted $p$-Laplace operator in outward H\"older cuspidal domains}

\author{Prashanta Garain, Valerii Pchelintsev, Alexander Ukhlov}

\begin{abstract}
In this paper, we will study Neumann $(p,q)$-eigenvalue problem for the weighted $p$-Laplace operator in outward H\"older cuspidal domains. The suggested method is based on the composition operators on weighted Sobolev spaces.
\end{abstract}
\maketitle

\section{Introduction}
In this paper, we consider the following weighted Neumann $(p,q)$-eigenvalue problem:
\begin{equation}\label{pr-m}
-\textrm{div}(|x|^{\alpha}|\nabla u|^{p-2}\nabla u)=\lambda \|u\|_{L_q(\Omega_{\gamma})}^{p-q}|u|^{q-2}u\text{ in }\Omega,\quad\frac{\partial u}{\partial\nu}=0\text{ on }\partial \Omega_{\gamma},
\end{equation}
in bounded outward cuspidal domains $\Omega_{\gamma}\subset\mathbb R^n$ with anisotropic H\"older $\gamma$-sin\-gu\-la\\-ri\-ti\-es {(introduced in \cite{GG94})}, where $\nu$ is an outward unit normal to the boundary of $\Omega_\gamma$. Here
{
\begin{equation}\label{domain}
\Omega_{\gamma}=\{ x=(x_1,x_2,\ldots,x_n)\in\mathbb R^n : 0<x_n<1, 0<x_i<g_i(x_n),\,i=1,2,\dots,n-1\},
\end{equation}}
where $g_i(t)=t^{\gamma_i}$, $\gamma_i\geq 1$, $0< t< 1$, are H\"older functions and we denote by $\gamma={\log (g_1(t)\cdot ... \cdot g_{n-1}(t))}{(\log t)}^{-1}+1$. {It is evident that $\gamma\geq n$.}
When $g_1=g_2=\dots=g_{n-1}$, we will say the domain $\Omega_{\gamma}$ is a domain with $\sigma$-H\"older singularity, $\sigma=(\gamma-1)/(n-1)$. {For $g_1(t)=g_2(t)=\ldots=g_{n-1}(t)=t$, we will use the notation $\Omega_n$ instead of $\Omega_{\gamma}$.}

The main results of this article include the case of Lipschitz domains $\Omega\subset\mathbb R^n$, in this case we can put $\gamma=n$.
Throughout the paper, we assume that $1<p<\alpha + \gamma$, where $-n<\alpha<n(p-1)$,\,$n\geq 2$, and $1<q<p^*$, where $p^*={\gamma p}/{(\alpha + \gamma -p)}$ unless otherwise mentioned.
The operator
$$
\textrm{div}(|x|^{\alpha} |\nabla u|^{p-2}\nabla u)
$$
is known as \textit{weighted $p$-Laplacian} (see, for example, \cite{FPR,HKM}). The weighted $p$-Laplacian is a non-linear operator when $p \neq 2$ and is linear when $p=2$.

We consider the weighted Neumann $(p,q)$-eigenvalue problem~\eqref{pr-m} in the weak formulation: a function $u$ solves this eigenvalue problem~\eqref{pr-m} iff $u \in W^{1}_p(\Omega_{\gamma}, |x|^{\alpha})$ (see definitions in Section 2) and
\begin{equation}\label{weak}
\int_{\Omega_{\gamma}} \langle |\nabla u|^{p-2}\nabla u, \nabla v\rangle |x|^{\alpha}~dx
= \lambda \|u\|_{L_q(\Omega_{\gamma})}^{p-q} \int_{\Omega_{\gamma}} |u|^{q-2}u v~dx
\end{equation}
for all $v \in W^{1}_p(\Omega_{\gamma},|x|^{\alpha})$.

The estimates of non-linear Neumann eigenvalues is a long-standing complicated problem \cite{PS51}. The classical result by Payne and Weinberger  \cite{PW} states that in convex domains $\Omega\subset\mathbb R^n$, $n\geq 2$
$$
\lambda_{2,2}(\Omega)\geq \frac{\pi^2}{d(\Omega)^2},
$$
where $d(\Omega)$ is a diameter of a convex domain $\Omega$.
Unfortunately in non-convex domains $\lambda_{2,2}(\Omega)$ can not be estimated in the terms of Euclidean diameters. It can be seen by considering a domain consisting of two identical squares connected by a thin corridor \cite{BCDL16}. In \cite{BCT9,BCT15} lower estimates involved the isoperimetric constant relative to $\Omega$ were obtained.

The approach to the spectral estimates in non-convex domains which is based on the composition operators on Sobolev spaces \cite{U93,VU04,VU05} and its applications to the Sobolev type weighted embedding theorems \cite{GG94,GU09} is given  in \cite{GU16,GU17}.

In the present work, we will prove the Min-Max Principle for the weighted Neumann $(p,q)$-eigenvalues and we obtain spectral estimates in bounded outward cuspidal domains with anisotropic H\"older $\gamma$-singularities. In addition, in the case $q=2$, we obtain regularity results for eigenfunctions corresponding to the weighted Neumann $(p,q)$-eigenvalue problem.

\section{Weighted Sobolev spaces and embedding theorems}

\subsection{Composition operators on weighted Sobolev spaces}

Let $E$ be a measurable subset of $\mathbb R^n$, $n \geq 2$, and $w: \mathbb R^n \to [0,+ \infty)$ be a locally integrable nonnegative function, i.e., a weight. The weighted Lebesgue space $L_p(E,w)$, $1\leq p<\infty$, is defined as a Banach space of locally integrable functions
$f: E \to \mathbb R$ endowed with the following norm:
\[
\|f\|_{L_{p}(E,w)}=\biggr(\int_{E}|f(x)|^{p} w(x)\, dx\biggr)^{\frac{1}{p}}.
\]
If $w=1$, we write $L_p(E,w)=L_p(E)$.

Let $\Omega$ be an open subset of $\mathbb R^n$, $n \geq 2$. We define the weighted Sobolev space $W^{1}_{p}(\Omega,w)$, $1\leq p<\infty$, as a normed space of locally integrable weakly differentiable functions
$f:\Omega\to\mathbb{R}$ endowed with the following norm:
\[
\|f\|_{W^1_p(\Omega,w)}=\|f\|_{L_{p}(\Omega)} + \|\nabla f\|_{L_{p}(\Omega,w)}.
\]
If $w=1$, we write $W^{1}_{p}(\Omega,w)=W^{1}_{p}(\Omega)$.
Note that in this definition the weight $w$ is only in the Sobolev part in accordance with the weighted Neumann $(p,q)$-eigenvalue problem (\ref{pr-m}).

The seminormed weighted Sobolev space $L^1_p(\Omega,w)$, $1\leq p<\infty$, is defined
as a space of locally integrable weakly differentiable functions
$f:\Omega\to\mathbb{R}$ endowed with the following seminorm:
\[
\|f\|_{L^1_p(\Omega,w)}=\|\nabla f\|_{L_{p}(\Omega,w)}.
\]
If $w=1$, we write $L^1_p(\Omega,w)=L^1_p(\Omega)$.

Remark that without additional restrictions on $w$, the space $W^1_p(\Omega,w)$ is not necessarily a Banach space (see, for example, \cite{Kufn}). However, if the weight $w : \mathbb R^n \to [0,+ \infty)$ satisfy the $A_p$-condition:
\[
\sup_{B \subset \mathbb R^n} \left(\frac{1}{|B|}\int_{B} w\right)
\left(\frac{1}{|B|}\int_{B} w^{1/(1-p)}\right)^{p-1}<+ \infty ,
\]
where $1<p<\infty$ and $|B|$ is the Lebesgue measure of the ball $B$, then $W^1_p(\Omega,w)$ is a Banach space and smooth functions of class $W^1_p(\Omega,w)$ are dense in $W^1_p(\Omega,w)$ \cite{GU09}. We remark that $w(x)=|x|^\alpha\in A_p$ iff $-n<\alpha<n(p-1)$, see \cite{HKM}.

The weighted Sobolev space $W^1_{p,\loc}(\Omega,w)$ is defined as follows: $f\in W^1_{p,\loc}(\Omega,w)$
iff $f\in W^1_p(U,w)$ for every open and bounded set $U\subset  \Omega$ such that
$\overline{U}  \subset \Omega$, where $\overline{U} $ is the closure of the set $U$.

In accordance with the non-linear potential theory, we consider the Sobolev spaces as Banach spaces of equivalence classes of functions up to a set of a weighted  $p$-capacity zero \cite{HKM,M}.

Let $\Omega$ and $\widetilde{\Omega}$ be domains in the Euclidean space $\mathbb R^n$.
Then a homeomorphism $\varphi:\Omega\to\widetilde{\Omega}$ induces a bounded composition
operator \cite{VU04,VU05}
\[
\varphi^{\ast}:L^1_p(\widetilde{\Omega},w)\to L^1_s(\Omega),\,\,\,1 \leq s\leq p<\infty,
\]
by the composition rule $\varphi^{\ast}(f)=f\circ\varphi$, if the composition $\varphi^{\ast}(f)\in L^1_s(\Omega)$
is defined quasi-everywhere in $\Omega$ and there exists a constant $K_{p,s}(\varphi;\Omega)<\infty$ such that
\[
\|\varphi^{\ast}(f)\|_{L^1_s(\Omega)}\leq K_{p,s}(\varphi;\Omega)\|f\|_{L^1_p(\widetilde{\Omega},w)}
\]
for any refined function $f\in L^1_p(\widetilde{\Omega},w)$.

Recall that if $\varphi : \Omega \to \widetilde{\Omega}$ is a Sobolev homeomorphism, then its formal Jacobi matrix $D \varphi(x)$ and its determinant (Jacobian) $J(x, \varphi)$ are well defined at almost all points $x \in \Omega$. The norm $|D \varphi(x)|$ of the
matrix $D \varphi(x)$ is the norm of the corresponding linear operator. The Sobolev homeomorphism $\varphi: \Omega \to \widetilde{\Omega}$ of class $W^1_{1,\loc}(\Omega)$ is called a mapping of finite weighted distortion \cite{GU09} if
\[
D\varphi=0\,\,\, \text{a.e. on the set}\,\,\, Z=\{x \in \Omega:|J(x,\varphi)|w(\varphi(x))=0\}.
\]
Recall that a mapping $\varphi:\Omega\to\widetilde{\Omega}$ possesses the Luzin $N$-property, if the image of a set of measure zero has measure zero.

Let $\Omega$ and $\widetilde{\Omega}$ be domains in the Euclidean space $\mathbb R^n$. Then
a homeomorphism $\varphi:\Omega\to \widetilde{\Omega}$ is called $w$-weighted $(p,s)$-quasiconformal, $1<s<p<\infty$,  if $\varphi$ belongs
to the Sobolev space $W^{1}_{1,\loc}(\Omega)$, has finite weighted distortion and
\begin{equation}\label{kps}
K_{p,s}(\varphi;\Omega)= \left(\int_{\Omega}\left(\frac{|D\varphi(x)|^p}{|J(x,\varphi)|w(\varphi(x))}\right)^{\frac{s}{p-s}}dx\right)^{\frac{p-s}{ps}}
<\infty.
\end{equation}

The following result gives an analytic characterization of composition operators on weighted seminormed Sobolev spaces \cite{UV08} (see also \cite{GU09,PU23}) and states that:

\begin{theorem}\label{CompTh}
Assume that $w\in A_p$ for some $1<p<\infty$. Let $\Omega$ and $\widetilde{\Omega}$ be domains in the Euclidean space $\mathbb R^n$. Then  a homeomorphism  $\varphi:\Omega\to{\widetilde{\Omega}}$ generates, by the composition rule $\varphi^{*}(f)=f \circ \varphi$, a bounded composition operator
\[
\varphi^{\ast}:L^{1}_{p}(\widetilde{\Omega},w)\to L^{1}_{s}(\Omega),\,\,\,1<s< p < \infty,
\]
if and only if $\varphi$ is a $w$-weighted $(p,s)$-quasiconformal mapping.
\end{theorem}
Now we give examples of $|x|^\alpha$-weighted $(p,s)$-quasiconformal mappings of Lipschitz domains onto bounded outward cuspidal domains with anisotropic H\"older $\gamma$-singularities \cite{GU09}.

Recall the definition of $\Omega_\gamma$ from \eqref{domain}. Also, we recall the notation $\Omega_n$ instead of $\Omega_\gamma$, when $g_1(t)=g_2(t)=\ldots=g_{n-1}(t)=t$. Define the mapping $\varphi_a:\Omega_n\to \Omega_{\gamma}$, $a>0$, from the Lipschitz convex domain $\Omega_n$ onto the outward cuspidal domain $\Omega_{\gamma}$ by the formula
\begin{equation}
\label{phi_a}
\varphi_a(x)=\left(\frac{x_1}{x_n}g_1^a(x_n),\ldots,\frac{x_{n-1}}{x_n}g_{n-1}^a(x_n),x^a_n\right).
\end{equation}

The following theorem refines the corresponding assertion from \cite{GU09}.

\begin{theorem}\label{S-P-ineq}
Let $-n<\alpha <n(p-1)$, $1<p<\alpha + \gamma$. Then {for $\max\{0,(n-p)/(\alpha + \gamma -p)\}<a<p(n-1)/(\alpha + \gamma -p)$}, the mapping $\varphi_a:\Omega_n\to \Omega_{\gamma}$,
is a $|x|^{\alpha}$-weighted (p,s)-quasiconformal mapping, for any
$1<s<{np}/{(a(\alpha+\gamma-p)+p)}$, of the Lipschitz convex domain $\Omega_n$ onto the bounded outward cuspidal domain with anisotropic H\"older $\gamma$-singularity $\Omega_{\gamma}$, possesses the Luzin $N$-property, and
\begin{multline}
\label{pq_est}
\|\varphi_a^{\ast}\|\leq K_{p,s}(\varphi_a;\Omega_n)\\
\leq \frac{\sqrt{\sum_{i=1}^{n-1}(a\gamma_i-1)^2+n-1+a^2}}{\sqrt[p]{a c_a}}\cdot \left(\frac{p-s}{np-s(a(\alpha+\gamma-p)+p)}\right)^{\frac{p-s}{ps}},
\end{multline}
{ for some constant $c_a>0$.}
\end{theorem}

\begin{proof}
In \cite{GU09,GU17} it was proved that for the mapping $\varphi_a:\Omega_n\to \Omega_{\gamma}$, $a>0$, defined by the formula (\ref{phi_a}) belongs to the Sobolev space $W^1_{1,\loc}(\Omega)$, the Jacobian $J(x,\varphi_a)=a x_n^{a-n}G^a(x_n)=a x_n^{a \gamma-n}>0$ in $\Omega_n$, and so $\varphi_a$ is a mapping of finite distortion \cite{VGR79}. In addition
\[
|D\varphi_a(x)| \leq x_n^{a-1} \sqrt{\sum_{i=1}^{n-1}(a\gamma_i-1)^2+n-1+a^2}.
\]
This estimate, gives, in particular, that
\begin{multline*}
\int_{\Omega_n}|D\varphi_a(x)|^n~dx\leq \left(\sum_{i=1}^{n-1}(a\gamma_i-1)^2+n-1+a^2\right)^{\frac{n}{2}}\int_{\Omega_n}x_n^{n(a-1)}~dx
\\=
\left(\sum_{i=1}^{n-1}(a\gamma_i-1)^2+n-1+a^2\right)^{\frac{n}{2}}\int_{0}^1\int_{0}^{x_n}... \int_{0}^{x_n} x_n^{n(a-1)}~dx_1 ...dx_{n-1}dx_n
\\=
\left(\sum_{i=1}^{n-1}(a\gamma_i-1)^2+n-1+a^2\right)^{\frac{n}{2}}\int_{0}^1 x_n^{n(a-1)+n-1}~dx_n<\infty,
\end{multline*}
if 
\begin{equation}
\label{lp}
n(a-1)+n-1>-1.
\end{equation}
The inequality \eqref{lp} holds for any $a>0$ and so the homeomorphism $\varphi_a$ belongs to the Sobolev space $L^1_n(\Omega)$ and possesses the Luzin $N$-property \cite{VGR79}.

Since $x_i \leq x_n$, $i=1,2,\ldots,n-1$, for all $x \in \Omega_n$, there exist constants $c_a, C_a>0$ such that
\[
c_a x_n^{a \alpha} \leq |\varphi_a(x)|^{\alpha} \leq C_a x_n^{a \alpha}.
\]

Let  $1<p<\alpha + \gamma$ and $\max\{0,(n-p)/(\alpha + \gamma -p)\}<a<p(n-1)/(\alpha + \gamma -p)$. We prove that  the integral
\begin{equation*}
K_{p,s}(\varphi_a;\Omega_n)= \left(\int_{\Omega_n}\left(\frac{|D\varphi_a(x)|^p}{|J(x,\varphi_a)|{|\varphi_a(x)|^\alpha})}\right)^{\frac{s}{p-s}}dx\right)^{\frac{p-s}{ps}}
\end{equation*}
is finite for any $1<s<{np}/{(a(\alpha+\gamma-p)+p)}$.

By estimates of $|D\varphi_a(x)|$, $|\varphi_a(x)|^{\alpha}$ and the value of $J(x,\varphi_a)$ we obtain, that
\begin{multline}
\label{est_norm}
K_{p,s}(\varphi_a;\Omega_n)=
\left(\int_{\Omega_n}\left(\frac{|D\varphi_a(x)|^p}{|J(x,\varphi_a)||\varphi_a(x)|^{\alpha}}\right)^{\frac{s}{p-s}}dx
\right)^{\frac{p-s}{ps}} \\
\leq \frac{\sqrt{\sum_{i=1}^{n-1}(a\gamma_i-1)^2+n-1+a^2}}{\sqrt[p]{a c_a}}
\left(\int_{\Omega_n} x_n^{\frac{(p(a-1)-a(\alpha + \gamma) +n)s}{p-s}}dx
\right)^{\frac{p-s}{ps}} \\
= \frac{\sqrt{\sum_{i=1}^{n-1}(a\gamma_i-1)^2+n-1+a^2}}{\sqrt[p]{a c_a}}
\left(\int_{0}^1 x_n^{\frac{(p(a-1)-a(\alpha + \gamma) +n)s}{p-s}+n-1}dx_n
\right)^{\frac{p-s}{ps}}<\infty,
\end{multline}
for any number $s$ such that 
\begin{equation}\label{ol1}
\frac{(p(a-1)-a(\alpha + \gamma) +n)s}{p-s}+n-1>-1.
\end{equation}
Now the inequality $a>\max\{\frac{n-p}{\alpha+\gamma-p},0\}$ gives us
$
\frac{np}{a(\alpha+\gamma-p)+p}<p
$
and the inequality
$
a<\frac{p(n-1)}{\alpha+\gamma-p}
$
implies that
$
\frac{np}{a(\alpha+\gamma-p)+p}>1.
$
Therefore \eqref{ol1} is true, since $
1<s<\frac{np}{a(\alpha+\gamma-p)+p}.
$
Hence the integral (\ref{est_norm}) is finite for any number $p$ such that $1<p<\alpha + \gamma$ and for any number $s$ such that 
$$
1<s<\frac{np}{a(\alpha+\gamma-p)+p}.
$$
So, by the definition, $\varphi_a:\Omega_n\to \Omega_{\gamma}$, $\max\left\{0,\frac{n-p}{\alpha + \gamma -p}\right\}<a<\frac{p(n-1)}{\alpha+\gamma-p}$,
is a $|x|^{\alpha}$-weighted $(p,s)$-quasiconformal mapping for any $1<p<\alpha + \gamma$ and $1<s<{np}/{(a(\alpha+\gamma-p)+p)}$.

In addition, by Theorem~\ref{CompTh} and the inequality (\ref{est_norm}), we have
\begin{multline*}
\|\varphi_a^{\ast}\|\leq K_{p,s}(\varphi_a;\Omega_n)
 \\
\leq \frac{\sqrt{\sum_{i=1}^{n-1}(a\gamma_i-1)^2+n-1+a^2}}{\sqrt[p]{a c_a}}
\left(\int_{0}^1 x_n^{\frac{(p(a-1)-a(\alpha + \gamma) +n)s}{p-s}+n-1}dx_n
\right)^{\frac{p-s}{ps}} \\
= \frac{\sqrt{\sum_{i=1}^{n-1}(a\gamma_i-1)^2+n-1+a^2}}{\sqrt[p]{a c_a}}\cdot \left(\frac{p-s}{np-s(a(\alpha+\gamma-p)+p)}\right)^{\frac{p-s}{ps}}.
\end{multline*}
{This completes the proof.}
\end{proof}

\begin{remark}
In conditions of Theorem~\ref{S-P-ineq} if 
$$
\frac{(p(a-1)-a(\alpha + \gamma) +n)s}{p-s}+n-1>0,
$$
that equivalent
$$
s<\frac{(n-1)p}{a(\alpha+\gamma-p)+p+1},
$$
then
$$
\|\varphi_a^{\ast}\|\leq K_{p,s}(\varphi_a;\Omega_n)\leq
 \frac{\sqrt{\sum_{i=1}^{n-1}(a\gamma_i-1)^2+n-1+a^2}}{\sqrt[p]{a c_a}}.
$$

\end{remark}

\subsection{Weighted Sobolev embeddings}

In this section we consider embedding of weighted Sobolev spaces into classical Lebesgue spaces in bounded outward cuspidal domains with anisotropic H\"older $\gamma$-singularities. Let $\Omega \subset \mathbb R^n$ be a bounded domain. Then $\Omega$ is called an $(r,s)$-Poincar\'e-Sobolev domain, $1\leq r,s \leq \infty$, if there exists a constant $C<\infty$, such that for any function $f\in W^1_s(\Omega)$, the $(r,s)$-Poincar\'e-Sobolev inequality
\[
\inf\limits_{c \in \mathbb R}\|f-c\|_{L_r(\Omega)}\leq C\|\nabla f\|_{L_s(\Omega)}
\]
holds. We denote by $B_{r,s}(\Omega)$ the best constant in this inequality.

On the first stage we prove the theorem on composition operators in normed Sobolev spaces $W^1_p(\widetilde{\Omega},w)$ and $W^1_s(\Omega)$. Let us recall from \cite{VU04,VU05}, that a homeomorphism $\varphi: \Omega\to \widetilde{\Omega}$ generates an inverse bounded composition operator, by the rule $\left(\varphi^{-1}\right)^{\ast}(g)=g\circ\varphi^{-1}$, on Lebesgue spaces
$$
\left(\varphi^{-1}\right)^{\ast}: L_r(\Omega)\to L_q(\widetilde{\Omega}), \,\,1<q<r<\infty,
$$
iff $\varphi$ possesses the Luzin $N$-property and 
\[
\int_{\Omega} |J(x,\varphi)|^{\frac{r}{r-q}}~dx < \infty.
\]

\begin{theorem}\label{emb}
Let $\Omega \subset \mathbb R^n$ be an $(r,s)$-Poincar\'e-Sobolev domain for some $1< s \leq r < \infty$. Let $w\in A_p$ for some $1<p<\infty$. Suppose that there exists a $w$-weighted $(p,s)$-quasiconformal mapping $\varphi$ of $\Omega$ onto the bounded domain $\widetilde{\Omega}\subset\mathbb R^n$ for some $s$ such that $1<s<p$ and $\varphi$ possesses the Luzin $N$-property.
If
\[
\int_{\Omega} |J(x,\varphi)|^{\frac{r}{r-q}}~dx < \infty, \,\,\text{for some}\,\, p \leq q \leq r,
\]
then the composition operator
\[
\varphi^{\ast}:W^1_p(\widetilde{\Omega},w) \to W^1_s(\Omega),
\]
is bounded.
\end{theorem}

\begin{proof}
Let a function $f$ belongs to the Sobolev space $W^1_p(\widetilde{\Omega},w)$. Since $\varphi: \Omega\to \widetilde{\Omega}$ is a $w$-weighted $(p,s)$-quasiconformal mapping, there exists a bounded composition operator 
\[
\varphi^{\ast}:L^1_p(\widetilde{\Omega},w) \to L^1_s(\Omega).
\]
Hence, the inequality
\begin{equation}\label{in1}
\|\varphi^{\ast}(f)\|_{L^1_s(\Omega)}\leq K_{p,s}(\varphi;\Omega) \|f\|_{L^1_p(\widetilde{\Omega},w)}
\end{equation}
holds for any function $f$ belongs to the Sobolev space $W^1_p(\widetilde{\Omega},w)$.

Let us denote the composition by the symbol $g=\varphi^{*}(f)$. Because
\[
\int_{\Omega} |J(x,\varphi)|^{\frac{r}{r-q}}~dx < \infty,
\]
and $\varphi$ possesses the Luzin $N$-property, the inverse composition operator $(\varphi^{-1})^*:L_r(\Omega)\to L_q(\widetilde{\Omega})$ is bounded \cite{VU04,VU05}, i.e., there exists a positive constant $A_{r,q}(\Omega)<\infty$, such that the inequality
\begin{equation}\label{in2}
\|(\varphi^{-1})^*g\|_{L_q(\widetilde{\Omega})} \leq A_{r,q}(\Omega) \|g\|_{L_r(\Omega)}
\end{equation}
holds for any function $g\in L_r(\Omega)$.

Since the domain $\Omega$ is an $(r,s)$-Poincar\'e-Sobolev domain
\begin{equation}\label{in21}
\inf\limits_{c \in \mathbb R}\|g-c\|_{L_r(\Omega)}\leq B_{r,s}(\Omega)\|g\|_{L^1_s(\Omega)},
\end{equation}
and by Theorem~\ref{CompTh} the composition operator
\[
\varphi^*:L^1_p(\widetilde{\Omega},w)\to L^1_s(\Omega)
\]
is bounded, we obtain the following inequalities: 
\begin{multline}
\label{in22}
\inf\limits_{c \in \mathbb R}\|f-c\|_{L_q(\widetilde{\Omega})}\\
\leq \left\|f-\frac{1}{|\Omega|}\int_{\Omega}|g|^{q-2}g~dx\right\|_{L_q(\widetilde{\Omega})}
\leq A_{r,q}(\Omega)\cdot\left\|g-\frac{1}{|\Omega|}\int_{\Omega}|g|^{q-2}g~dx\right\|_{L_r(\Omega)}\\ 
= A_{r,q}(\Omega)\cdot \inf\limits_{c \in \mathbb R}\|g-c\|_{L_r(\Omega)} 
\leq A_{r, q}(\Omega) B_{r,s}(\Omega)\|g\|_{L^1_s(\Omega)}\\
\leq A_{r, q}(\Omega) B_{r,s}(\Omega) K_{p,s}(\varphi;\Omega) \|f\|_{L^1_p(\widetilde{\Omega},w)},
\end{multline}
where $K_{p,s}(\varphi;\Omega)$ is defined in \eqref{kps}.

The H\"older inequality implies the following estimate for any $c\in\mathbb{R}$:
\begin{multline}
\label{in4}
|c|=|\widetilde{\Omega}|^{-\frac{1}{p}} \|c\|_{L_p(\widetilde{\Omega})}\\
\leq |\widetilde{\Omega}|^{-\frac{1}{p}} \left(\|f\|_{L_p(\widetilde{\Omega})}+\|f-c\|_{L_p(\widetilde{\Omega})}\right) \\
\leq |\widetilde{\Omega}|^{-\frac{1}{p}} \|f\|_{L_p(\widetilde{\Omega})}
+ |\widetilde{\Omega}|^{-\frac{1}{q}} \|f-c\|_{L_q(\widetilde{\Omega})}.
\end{multline}

Because $s \leq r$, we have
\begin{equation}\label{in5}
\begin{split}
\|g\|_{L_s(\Omega)} &\leq \|c\|_{L_s(\Omega)} + \|g-c\|_{L_s(\Omega)}
\leq |c||\Omega|^{\frac{1}{s}} +|\Omega|^{\frac{r-s}{rs}}\|g-c\|_{L_r(\Omega)} \\
&\leq \left(|\widetilde{\Omega}|^{-\frac{1}{p}} \|f\|_{L_p(\widetilde{\Omega})}
+ |\widetilde{\Omega}|^{-\frac{1}{q}} \|f-c\|_{L_q(\widetilde{\Omega})}\right)|\Omega|^{\frac{1}{s}}
+|\Omega|^{\frac{r-s}{rs}}\|g-c\|_{L_r(\Omega)},
\end{split}
\end{equation}
where we have applied \eqref{in4}.
Now in \eqref{in5}, we combine the previous inequalities \eqref{in1}, \eqref{in21} and \eqref{in22} and finally obtain
\begin{multline*}
\|g\|_{L_s(\Omega)} \leq |\Omega|^{\frac{1}{s}} |\widetilde{\Omega}|^{-\frac{1}{p}} \|f\|_{L_p(\widetilde{\Omega})} \\
+ A_{r, q}(\Omega) B_{r,s}(\Omega) K_{p,s}(\varphi;\Omega) |\Omega|^{\frac{1}{s}} |\widetilde{\Omega}|^{-\frac{1}{q}}
\|f\|_{L^1_p(\widetilde{\Omega},w)} \\
+ B_{r,s}(\Omega) K_{p,s}(\varphi;\Omega) |\Omega|^{\frac{r-s}{rs}} \|f\|_{L^1_p(\widetilde{\Omega},w)}.
\end{multline*}

Therefore the composition operator
\[
\varphi^*:W^1_p(\widetilde{\Omega},w)\to W^1_s(\Omega)
\]
is bounded, which proves the theorem.
\end{proof}
Now we prove the weighted Sobolev embedding theorem in bounded outward cuspidal domains with anisotropic H\"older $\gamma$-singularities.

\begin{theorem}\label{thmemb}
Let $\Omega_{\gamma}\subset\mathbb R^n$ be a domain with anisotropic H\"older $\gamma$-sin\-gu\-larities.
Suppose $\max\big\{-n,\frac{p(n-\gamma)}{n}\big\}< \alpha <n(p-1)$ and $1<p<\alpha + \gamma$. Then the embedding operator
$$
i_{\Omega_{\gamma}}:W^1_p(\Omega_{\gamma},|x|^{\alpha})\hookrightarrow L_{q}(\Omega_{\gamma})
$$
is compact for any $1<q<p^*$, where $p^*={\gamma p}/{(\alpha + \gamma -p)}$.
\end{theorem}

\begin{proof}
To prove the theorem we will apply Theorem~\ref{emb}.
First, we will find some suitable $r_0$ and $s_0$ to verify Theorem~\ref{emb}. To this end, we remark that $\frac{n-p}{\alpha+\gamma-p}<\frac{n}{\gamma}$, since $\alpha>\frac{p(n-\gamma)}{n}$. We define the interval
$$
\mathcal{I}:=\Big(\max\Big\{0,\frac{n-p}{\alpha+\gamma-p}\Big\},\min\Big\{\frac{n}{\gamma},\frac{p(n-1)}{\alpha+\gamma-p}\Big\}\Big).
$$
Let us fix $a\in\mathcal{I}$ and we set
$$
t=\frac{np}{a(\alpha+\gamma-p)+p}.
$$
Note that since $a\in\mathcal{I}$, we obtain $1<t<\min\{p,n\}$. Let us fix $q\in[p,p^*)$. We observe that as $s\to t$, where $s\in(1,t)$, we get $\frac{a\gamma s}{n-s}\to p^*$. Hence, taking into account that $p^*-q>0$, there exists $s_0\in(1,t)$ such that 
\begin{equation}\label{inenew}
\frac{a\gamma s_0}{n-s_0}-p^*>q-p^*.
\end{equation}
The above inequality \eqref{inenew} yields that 
\begin{equation}\label{inenew2}
\frac{nq}{a\gamma}<\frac{ns_0}{n-s_0}.
\end{equation}
Since $a<\frac{n}{\gamma}$, we get $\frac{nq}{a\gamma}>q$ and so
\eqref{inenew2} gives that
\begin{equation}\label{inenew3}
q<\frac{nq}{a\gamma}<\frac{ns_0}{n-s_0}.
\end{equation}
Therefore, we choose $r_0\in\big(\frac{nq}{a\gamma},\frac{ns_0}{n-s_0}\big)$ such that $s_0<r_0$. Hence $q<\frac{a\gamma}{n}r_0<r_0$.

Now, we verify Theorem~\ref{emb} in the following four steps: First, We note that $|x|^\alpha\in A_p$, since $-n<\alpha<n(p-1)$. Now, remark that $\Omega_n$ is a Lipschitz domain and so is an $(r_0,s_0)$-Poincar\'e-Sobolev domain. Next, by Theorem~\ref{S-P-ineq}, the mapping $\varphi_a:\Omega_n\to \Omega_{\gamma}$, $a\in\mathcal{I}$ defined by the formula~(\ref{phi_a}) is a $|x|^\alpha$-weighted $(p,s_0)$-quasiconformal mapping for $1<p<\alpha + \gamma$, $1<s_0<t$. In addition, by Theorem~\ref{S-P-ineq} the mapping $\varphi_a$ possesses the Luzin $N$-property.
We observe that
\begin{multline*}
J_a=\int_{\Omega_n}|J(x,\phi_a)|^\frac{r_0}{r_0-q}\,dx\\
=a^{\frac{r_0}{r_0-q}} \int_{0}^1
\int_{0}^{x_n} \ldots \int_{0}^{x_n} x_n^{\frac{(a\gamma -n)r_0}{r_0-q}}dx_1 \ldots dx_n
=a^{\frac{r_0}{r_0-q}} \int_{0}^1 x_n^{\frac{(a\gamma -n)r_0}{r_0-q}+n-1}dx_n.
\end{multline*}
Since $q<\frac{a\gamma}{n}r_0$, we obtain
$$
\frac{(a\gamma -n)r_0}{r_0-q}+n-1>-1.
$$
Hence we get $J_a<\infty$. Therefore, applying Theorem~\ref{emb}, we deduce that the embedding operator
$$
\varphi^{\ast}:W^1_p(\Omega_{\gamma},|x|^{\alpha})\hookrightarrow W^{1}_{s_0}(\Omega_n)
$$
is bounded. Now the embedding operator
$$
i_{\Omega_{\gamma}}:W^1_p(\Omega_{\gamma},|x|^{\alpha})\hookrightarrow L_{q}(\Omega_{\gamma})
$$
is compact for our chosen $q\in[p,p^*)$, since $i_{\Omega_{\gamma}}=(\varphi^{-1})^{\ast}\circ  i_{\Omega_n}\circ \varphi^{\ast}$, where $i_{\Omega_n}:W_{s_0}^1(\Omega_n)\hookrightarrow L_{r_0}(\Omega_n)$ is the compact embedding operator (since $r_0<\frac{ns_0}{n-s_0}$) and $(\varphi^{-1})^{\ast}:L_{r_0}(\Omega_n)\hookrightarrow L_{q}(\Omega_{\gamma})$ is the bounded inverse composition operator on Lebesgue spaces. Since, $q$ is arbitrary, the embedding $i_{\Omega_\gamma}$ is compact for any $q\in[p,p^*)$. Therefore, taking into account the continuous embedding $L_\beta(\Omega_\gamma)\hookrightarrow L_\alpha(\Omega_\gamma)$ for $1<\alpha<\beta$ the proof follows for any $q\in(1,p^*)$. This completes the proof.
\end{proof}

\section{Neumann eigenvalue problem}

\subsection{Min-Max Principle} We assume that $\max\big\{-n,\frac{p(n-\gamma)}{n}\big\}< \alpha <n(p-1)$, $1<p<\alpha + \gamma$ and $1<q<p^*$.
Now we prove the Min-Max Principle for the first non-trivial Neumann $(p,q)$-eigenvalue~\eqref{weak}.
To this end, we follow the proof of \cite[Lemma 2]{CHP} and first obtain the following auxiliary result.

\begin{lemma}\label{lem1}
Let $v\in W^{1}_p(\Omega_{\gamma},|x|^{\alpha})\setminus\{0\}$ be such that $\int_{\Omega_{\gamma}}|v|^{q-2}v\,dx=0$. Then there exists a constant $C=C(\Omega_{\gamma})>0$ such that
$$
\|v\|_{L_p(\Omega_{\gamma})}\leq C\|\nabla v\|_{L_p(\Omega_{\gamma},|x|^{\alpha})}.
$$
\end{lemma}

\begin{proof}
By contradiction, suppose for every $n\in\mathbb{N}$, there exists $v_n\in W^{1}_p(\Omega_{\gamma},|x|^{\alpha})\setminus\{0\}$ such that $\int_{\Omega_{\gamma}}|v_n|^{q-2}v_n\,dx=0$ and
\begin{equation}\label{cn}
\|v_n\|_{L_p(\Omega_{\gamma})}>n\|\nabla v_n\|_{L_p(\Omega_{\gamma},|x|^{\alpha})}.
\end{equation}
Without loss of generality, let us assume that $\|v_n\|_{L_p(\Omega_{\gamma})}=1$. If not, we define
$$
u_n=\frac{v_n}{\|v_n\|_{L_p(\Omega_{\gamma})}},
$$
then $\|u_n\|_{L_p(\Omega_{\gamma})}=1$ and \eqref{cn} holds for $u_n$ and also $\int_{\Omega_{\gamma}}|u_n|^{q-2}u_n\,dx=0$. By \eqref{cn}, since $\|\nabla v_n\|_{L_p(\Omega_{\gamma},|x|^{\alpha})}\to 0$ as $n\to\infty$, we have $\{v_n\}$ is uniformly bounded in $W^{1}_p(\Omega_{\gamma},|x|^{\alpha})$.

Hence by Theorem \ref{thmemb}, there exists $v\in W^{1}_p(\Omega_{\gamma},|x|^{\alpha})$ such that
$$
v_n{\rightharpoonup} v\text{ weakly\,in }W^{1}_p(\Omega_{\gamma},|x|^{\alpha}),\quad v_n\to v\text{ strongly\,in }L_q(\Omega_{\gamma}),\quad\forall \,1<q<p^{*}
$$
and
$$
\nabla v_n{\rightharpoonup}\nabla v\text{ weakly\,in } L_p(\Omega_{\gamma},|x|^{\alpha}).
$$
Hence, by \cite[Theorem~4.9]{B10} there exists a function $g\in L_q(\Omega_{\gamma})$ such that
$$
|v_n|\leq g\text{ a.e.\, in }\Omega_{\gamma}.
$$

Since $\|\nabla v_n\|_{L_p(\Omega_{\gamma},|x|^{\alpha})}\to 0$ as $n\to\infty$, we have $\nabla v_n{\rightharpoonup} 0$ weakly in $L_p(\Omega_{\gamma},|x|^{\alpha})$, hence $\nabla v=0$ a.e. in $\Omega_{\gamma}$, which gives that $v=constant$ a.e. in $\Omega_{\gamma}$. This combined with the fact that
$$
0=\lim_{n\to\infty}\int_{\Omega_{\gamma}}|v_n|^{q-2}v_n\,dx=\int_{\Omega_{\gamma}}|v|^{q-2}v\,dx
$$
gives that $v=0$ a.e. in $\Omega_{\gamma}$. This contradicts the hypothesis that $\|v_n\|_{L_p(\Omega_{\gamma})}=1$. This completes the proof.
\end{proof}

\begin{theorem}\label{minthm}
There exists $u\in W^1_p(\Omega_{\gamma},|x|^{\alpha})\setminus\{0\}$ such that $\int_{\Omega_{\gamma}}|u|^{q-2}u\,dx=0$. Moreover,
\begin{multline*}
\lambda_{p,q}(\Omega_{\gamma},|x|^{\alpha})= \\
\inf \left\{\frac{\|\nabla u\|^p_{L_p(\Omega_{\gamma},|x|^{\alpha})}}{\|u\|^p_{L_q(\Omega_{\gamma})}} : u \in W^1_p(\Omega_{\gamma},|x|^{\alpha}) \setminus \{0\},
\int_{\Omega_{\gamma}} |u|^{q-2}u~dx=0 \right\}\\
=\frac{\|\nabla u\|^p_{L_p(\Omega_{\gamma},|x|^{\alpha})}}{\|u\|^p_{L_q(\Omega_{\gamma})}}.
\end{multline*}
Further, $\lambda_{p,q}(\Omega_{\gamma},|x|^{\alpha})>0.$
\end{theorem}

\begin{proof}
Let $n\in\mathbb{N}$ and define the functionals $G: W^1_p(\Omega_{\gamma},|x|^{\alpha})\to\mathbb{R}$ by
$$
G(v)=\int_{\Omega_{\gamma}}|v|^{q-2}v\,dx,
$$
and $H_\frac{1}{n}: W^1_p(\Omega_{\gamma},|x|^{\alpha})\to\mathbb{R}$ by
$$
H_\frac{1}{n}(v)=\|\nabla v\|_{L_p(\Omega_{\gamma},|x|^{\alpha})}^p-(\lambda_{p,q}+\frac{1}{n})\|v\|_{L_q(\Omega_{\gamma})}^p,
$$
where we denoted $\lambda_{p,q}(\Omega_{\gamma},|x|^{\alpha})$ by $\lambda_{p,q}$.
By the definition of infimum, for every $n\in\mathbb{N}$, there exists a function $u_n\in W^1_p(\Omega_{\gamma},|x|^{\alpha})\setminus\{0\}$ such that
$$
\int_{\Omega_{\gamma}}|u_n|^{q-2}u_n\,dx=0\,\,\text{ and}\,\, H_\frac{1}{n}(u_n)<0.
$$
Without loss of generality, let us assume that $\|\nabla u_n\|_{L_p(\Omega_{\gamma},|x|^{\alpha})}=1$. By Lemma \ref{lem1}, the sequence $\{u_n\}$ is uniformly bounded in $W^1_p(\Omega_{\gamma},|x|^{\alpha})$.
Hence, by Theorem \ref{thmemb}, because the embedding operator
$$
W^1_p(\Omega_{\gamma},|x|^{\alpha})\hookrightarrow L_q(\Omega_{\gamma}),\quad 1<q<p^{*},
$$
is compact, there exists $u\in W^1_p(\Omega_{\gamma},|x|^{\alpha})$ such that $u_n{\rightharpoonup} u$ weakly in $W^1_p(\Omega_{\gamma},|x|^{\alpha})$, $u_n\to u$ strongly in $L_q(\Omega_{\gamma})$ and $\nabla u_n{\rightharpoonup} \nabla u$ weakly in $L_p(\Omega_{\gamma},|x|^{\alpha})$.

Hence, by \cite[Theorem~4.9]{B10} there exists a function $g\in L_q(\Omega_{\gamma})$ such that
$$
|u_n|\leq g\,\,\text{ a.e. in}\,\, \Omega_{\gamma}.
$$

Since $|u_n|\leq g$ a.e. in $\Omega_{\gamma}$ and $u_n\to u$ a.e. in $\Omega_{\gamma}$, then
$$
||u_n|^{q-2}u_n|\leq |u_n|^{q-1}\leq |g|^{q-1}\in L_{q'}(\Omega_{\gamma}).
$$
So, by the Lebesgue Dominated Convergence Theorem (see, for example, \cite{Fe69}), it follows that
$$
0=\lim_{n\to\infty}\int_{\Omega_{\gamma}}|u_n|^{q-2}u_n\,dx=\int_{\Omega_{\gamma}}|u|^{q-2}u\,dx.
$$
So,
$$
\int_{\Omega_{\gamma}}|u|^{q-2}u\,dx=0.
$$
Due to $H_\frac{1}{n}(u_n)<0$, we obtain
\begin{equation}\label{1}
\|\nabla u_n\|_{L_p(\Omega_{\gamma},|x|^{\alpha})}^p-(\lambda_{p,q}+\frac{1}{n})\|u_n\|_{L_q(\Omega_{\gamma})}^p<0.
\end{equation}
Since $\nabla u_n{\rightharpoonup} \nabla u$ weakly in $L_p(\Omega_{\gamma},|x|^{\alpha})$, by the weak lower semicontinuity of norm, we have
$$
\|\nabla u\|^p_{L_p(\Omega_{\gamma},|x|^{\alpha})}\leq \lim\inf_{n\to\infty}\|\nabla u_n\|^p_{L_p(\Omega_{\gamma},|x|^{\alpha})}=1.
$$
So, by passing to the limit in \eqref{1}, we get
$$
\lambda_{p,q}\geq \frac{\|\nabla u\|^p_{L_p(\Omega_{\gamma},|x|^{\alpha})}}{\|u\|^p_{L_q(\Omega_{\gamma})}}.
$$
Therefore, by the definition of $\lambda_{p,q}$, we obtain
$$
\lambda_{p,q} = \frac{\|\nabla u\|^p_{L_p(\Omega_{\gamma},|x|^{\alpha})}}{\|u\|^p_{L_q(\Omega_{\gamma})}}.
$$
Now, since $H_{\frac{1}{n}}(u_n)<0$ and $\|\nabla u_n\|^p_{L_p(\Omega_{\gamma},|x|^{\alpha})}=1$, we have
$$
1-(\lambda_{p,q}+\frac{1}{n})\|u_n\|_{L_q(\Omega_{\gamma})}^p<0.
$$
Letting $n\to\infty$, we get
$$
\|u\|^p_{L_q(\Omega_{\gamma})}\lambda_{p,q}\geq 1,
$$
which gives $\lambda_{p,q}>0$ and $u\neq 0$ a.e. in $\Omega$.
\end{proof}

\section{Poincar\'e-Sobolev inequalities for functions of $W^1_p(\Omega,w)$}

The $w$-weighted $(p,s)$-quasiconformal mappings allow us to "transfer" the Poincar\'e-Sobolev inequalities from
one domain to another.

\begin{theorem}\label{SPineq}
Suppose $w\in A_p$ for some $1<p<\infty$. Let a bounded domain $\Omega \subset \mathbb R^n$ be a $(r,s)$-Poincar\'e-Sobolev domain, for some $1<s \leq r<\infty$, $1<s<p$, and there exists a $w$-weighted $(p,s)$-quasiconformal homeomorphism $\varphi:\Omega\to\widetilde{\Omega}$ of a domain $\Omega$ onto a bounded domain $\widetilde{\Omega}$, which possesses the Luzin $N$-property such that
\[
M_{r,q}(\Omega)=\biggr(\int_{\Omega}\left|J(x,\varphi)\right|^{\frac{r}{r-q}}\, dx\biggr)^{\frac{r-q}{rq}}< \infty
\]
for some $1\leq q<r$. Then in the domain $\widetilde{\Omega}$ the weighted $(q,p)$-Poincar\'e-Sobolev inequality
\[
\inf\limits_{c \in \mathbb R}\biggr(\int_{\widetilde{\Omega}}|f(y)-c|^q dy\biggr)^{\frac{1}{q}} \leq B_{q,p}(\widetilde{\Omega},w)
\biggr(\int_{\widetilde{\Omega}}|\nabla f(y)|^p w(y) dy\biggr)^{\frac{1}{p}},\,\,\, f\in W^{1,p}(\widetilde{\Omega},w),
\]
holds with the best constant $B_{q,p}(\widetilde{\Omega},w)$, such that
\[
B_{q,p}(\widetilde{\Omega},w)\leq K_{p,s}(\varphi;\Omega)M_{r,q}(\Omega)B_{r,s}(\Omega).
\]
Here $B_{r,s}(\Omega)$ is the best constant in the $(r,s)$-Poincar\'e-Sobolev inequality in the
domain $\Omega$ and $K_{p,s}(\varphi;\Omega)$ is as defined in \eqref{kps}.
\end{theorem}

\begin{proof}
Let $f\in W^{1}_{p}(\widetilde{\Omega},w)$. By the conditions of the theorem there exists a $w$-weighted $(p,s)$-quasiconformal homeomorphism $\varphi:\Omega\to\widetilde{\Omega}$. By Theorem~\ref{CompTh} the composition operator
\[
\varphi^*: L^1_p(\widetilde{\Omega},w) \to L^1_s(\Omega)
\]
is bounded. Since the bounded domain $\Omega$ is an $(r,s)$-Poincar\'e-Sobolev domain then a function $g=\varphi^*(f)\in W^1_s(\Omega)$.

Taking into account the change of variable formula for the Lebesgue integral and the H\"older's inequality we get
\begin{equation}\label{ne1}
\begin{split}
\inf\limits_{c \in \mathbb R}\biggr(\int_{\widetilde{\Omega}}|f(y)-c|^q dy\biggr)^{\frac{1}{q}}
&= \inf\limits_{c \in \mathbb R}\biggr(\int_{{\Omega}}|f(\varphi(x))-c|^q |J(x,\varphi)| dx\biggr)^{\frac{1}{q}} \\
&\leq \biggr(\int_{{\Omega}} |J(x,\varphi)|^{\frac{r}{r-q}} dx\biggr)^{\frac{r-q}{rq}}
\inf\limits_{c \in \mathbb R}\biggr(\int_{{\Omega}}|f(\varphi(x))-c|^r dx\biggr)^{\frac{1}{r}} \\
&= M_{r,q}(\Omega) \inf\limits_{c \in \mathbb R}\biggr(\int_{{\Omega}}|g(x)-c|^r dx\biggr)^{\frac{1}{r}}.
\end{split}
\end{equation}

Since the domain $\Omega$ is an $(r,s)$-Poincar\'e-Sobolev domain the following inequality holds:
\begin{equation}\label{ne2}
\inf\limits_{c \in \mathbb R}\biggr(\int_{{\Omega}}|g(x)-c|^r dx\biggr)^{\frac{1}{r}}
\leq B_{r,s}(\Omega) \biggr(\int_{{\Omega}}|\nabla g(x)|^s dx\biggr)^{\frac{1}{s}}.
\end{equation}
Combining two previous inequalities \eqref{ne1} and \eqref{ne2}, we have
\[
\inf\limits_{c \in \mathbb R}\biggr(\int_{\widetilde{\Omega}}|f(y)-c|^q dy\biggr)^{\frac{1}{q}}
\leq M_{r,q}(\Omega) B_{r,s}(\Omega) \biggr(\int_{{\Omega}}|\nabla g(x)|^s dx\biggr)^{\frac{1}{s}}.
\]
By Theorem~\ref{CompTh}
\[
\|g\|_{L^1_s(\Omega)} \leq K_{p,s}(\varphi;\Omega) \|f\|_{L^1_p(\widetilde{\Omega},w)}.
\]
Finally we obtain
\[
\inf\limits_{c \in \mathbb R}\biggr(\int_{\widetilde{\Omega}}|f(y)-c|^q dy\biggr)^{\frac{1}{q}}
\leq K_{p,s}(\varphi;\Omega) M_{r,q}(\Omega) B_{r,s}(\Omega)
\biggr(\int_{\widetilde{\Omega}}|\nabla f(y)|^p w(y) dy\biggr)^{\frac{1}{p}}.
\]

It means that
\[
B_{q,p}(\widetilde{\Omega},w)\leq K_{p,s}(\varphi;\Omega)M_{r,q}(\Omega)B_{r,s}(\Omega).
\]
\end{proof}

\section{Spectral estimates in H\"older singularities domains}
By Theorem~\ref{thmemb}, the embedding operator $W^1_p(\Omega_{\gamma},|x|^{\alpha})\hookrightarrow L_{q}(\Omega_{\gamma})$
is compact for any $1<q<p^*$ and by Theorem \ref{minthm}, the first non-trivial Neumann eigenvalue $\lambda_{p,q}(\Omega_{\gamma},|x|^{\alpha})$, $\max\big\{-n,\frac{p(n-\gamma)}{n}\big\}< \alpha <n(p-1)$, can be characterized as
\begin{multline*}
\lambda_{p,q}(\Omega_{\gamma},|x|^{\alpha})\\
=\inf \left\{\frac{\|\nabla u\|^p_{L_p(\Omega_{\gamma},|x|^{\alpha})}}{\|u\|^p_{L_q(\Omega_{\gamma})}} : u \in W^1_p(\Omega_{\gamma},|x|^{\alpha}) \setminus \{0\},
\int_{\Omega_{\gamma}} |u|^{q-2}u~dx=0 \right\}.
\end{multline*}
Furthermore, $\lambda_{p,q}(\Omega_{\gamma},|x|^{\alpha})^{-\frac{1}{p}}$ is equal to the best constant $B_{q,p}(\Omega_{\gamma},|x|^{\alpha})$ in the weighted $(q,p)$-Poincar\'e-Sobolev inequality
\begin{equation*}
\inf\limits_{c \in \mathbb R}\left(\int_{\Omega_{\gamma}} |f(x)-c|^q~dx\right)^{\frac{1}{q}}\\
\leq B_{q,p}(\Omega_{\gamma}, |x|^{\alpha})
\left(\int_{\Omega_{\gamma}} |\nabla f(x)|^p |x|^{\alpha}~dx\right)^{\frac{1}{p}}
\end{equation*}
for every function $f \in W^{1}_{p}(\Omega_{\gamma},|x|^{\alpha})$.

\begin{theorem}\label{eigen}
Let
\[
\Omega_{\gamma}:=\left\{x=(x_1,x_2,\ldots,x_n)\in\mathbb R^n: 0<x_n<1, 0<x_i<x_n^{\gamma_i}, i=1,2,\ldots,n-1\right\},
\]
$\gamma_i\geq 1$, $\gamma :=1+\sum_{i=1}^{n-1}\gamma_i$, be domains with anisotropic H\"older
$\gamma$-singularities as defined in \eqref{domain}.

Let $\max\big\{-n,\frac{p(n-\gamma)}{n}\big\}< \alpha <n(p-1)$, $1<p<\alpha + \gamma$and $1<q<p^*=p\gamma/(\alpha+\gamma-p)$ and $a\in \mathcal{I}:=\Big(\max\Big\{0,\frac{n-p}{\alpha+\gamma-p}\Big\},\min\Big\{\frac{n}{\gamma},\frac{p(n-1)}{\alpha+\gamma-p}\Big\}\Big)$. Then, we have
\begin{multline*}
\frac{1}{\lambda_{p,q}(\Omega_{\gamma},|x|^{\alpha})} \leq a^{\frac{p}{q}-1} c_a^{-1}
\left(\sum_{i=1}^{n-1}(a\gamma_i-1)^2+n-1+a^2\right)^{\frac{p}{2}} B_{r,s}^p(\Omega_n) \\
\times \left(\frac{p-s}{np-s(a(\alpha+\gamma-p)+p)}\right)^{\frac{p-s}{s}}
\left(\frac{r-q}{a\gamma r -nq} \right)^{\frac{(r-q)p}{rq}},
\end{multline*}
for some $r,s$ depending upon $a$ such that $s<\min\{p,n\}$ and $1<s<r<{ns}/{(n-s})$. Here $B_{r,s}(\Omega_n)$ is the best constant in the $(r,s)$-Poincar\'e-Sobolev inequality in the domain $\Omega_n$ and the constant $c_a$ defined in Theorem~\ref{S-P-ineq}.
\end{theorem}

\begin{proof}
Let $\max\big\{-n,\frac{p(n-\gamma)}{n}\big\}< \alpha <n(p-1)$, $1<p<\alpha + \gamma$, $a\in \mathcal{I}$ and $1<q<p^*=p\gamma/(\alpha+\gamma-p)$. Then proceeding along the lines of the proof of Theorem \ref{thmemb}, we get the existence of $r,s$ depedning on $a$ such that $1<s<r<\frac{ns}{n-s}$, $s<\frac{np}{a(\alpha+\gamma-p)+p}$, $q<\frac{a\gamma}{n}r<r$. Now, we will verify the conditions of Theorem~\ref{SPineq} to prove the result.
To this end, first we observe that $|x|^\alpha\in A_p$ for the given range of $\alpha$. Next, by Theorem~\ref{S-P-ineq} the mapping $\varphi_a:\Omega_n\to \Omega_{\gamma}$ is a $|x|^{\alpha}$-weighted $(p,s)$-quasiconformal mapping for $1<p<\alpha + \gamma$, $1<s<{np}/{(a(\alpha+\gamma-p)+p)}$ and has the Luzin $N$-property.
Since $\varphi_a$ is a $|x|^{\alpha}$-weighted $(p,s)$-quasiconformal mapping, by Theorem \ref{S-P-ineq}, the constant $K_{p,s}(\varphi_a;\Omega_n)$ is finite and satisfies the estimate
\begin{multline}\label{est1}
K_{p,s}(\varphi_a;\Omega_n) \\
\leq \frac{\sqrt{\sum_{i=1}^{n-1}(a\gamma_i-1)^2+n-1+a^2}}{\sqrt[p]{a c_a}}\cdot \left(\frac{p-s}{np-s(a(\alpha+\gamma-p)+p)}\right)^{\frac{p-s}{ps}}.
\end{multline}

The domain $\Omega_n$ is a Lipschitz domain and so is an $(r,s)$-Poincar\'e-Sobolev domain, i.e. $B_{r,s}(\Omega_n)<\infty$.

Next, we observe that
\begin{multline}
\label{est2}
M_{r,q}(\Omega_n)= \biggr(\int_{\Omega_n}\left|J(x,\varphi_a)\right|^{\frac{r}{r-q}}\, dx\biggr)^{\frac{r-q}{rq}}
= a^{\frac{1}{q}}
\biggr(\int_{\Omega_n}\left(x_n^{a\gamma -n}\right)^{\frac{r}{r-q}}\, dx\biggr)^{\frac{r-q}{rq}} \\
= a^{\frac{1}{q}}
\biggr(\int_{0}^1\left(x_n^{a\gamma -n}\right)^{\frac{r}{r-q}}
\biggr(\int_{0}^{x_n}\,dx_1 \ldots \int_{0}^{x_n}\,dx_{n-1}\biggr)
dx_n\biggr)^{\frac{r-q}{rq}} \\
= a^{\frac{1}{q}}
\biggr(\int_{0}^1x_n^{\frac{(a\gamma -n)r}{r-q}+n-1} dx_n \biggr)^{\frac{r-q}{rq}}
= a^{\frac{1}{q}}
\biggr(\frac{r-q}{a\gamma r -nq} \biggr)^{\frac{r-q}{rq}}<\infty,
\end{multline}
since $1<q<\frac{a \gamma}{n}r$.

So, the conditions of Theorem~\ref{SPineq} are fulfilled. Therefore, noting that $\lambda_{p,q}^{-1}(\Omega_{\gamma},|x|^{\alpha})=B_{q,p}^p(\Omega_{\gamma},|x^{\alpha}|)$, by Theorem~\ref{SPineq}, we obtain
\begin{multline*}
\frac{1}{\lambda_{p,q}(\Omega_{\gamma},|x|^{\alpha})}\leq K_{p,s}^p(\varphi_a;\Omega_n)M_{r,q}^p(\Omega_n)B^p_{r,s}(\Omega_n) \\
\leq a^{\frac{p}{q}-1} c_a^{-1}
\left(\sum_{i=1}^{n-1}(a\gamma_i-1)^2+n-1+a^2\right)^{\frac{p}{2}} B_{r,s}^p(\Omega_n) \\
\times \left(\frac{p-s}{np-s(a(\alpha+\gamma-p)+p)}\right)^{\frac{p-s}{s}}
\left(\frac{r-q}{a\gamma r -nq} \right)^{\frac{(r-q)p}{rq}},
\end{multline*}
where we have also used the estimates \eqref{est1} and \eqref{est2}. 
\end{proof}

\section{Existence results}

Throughout this section, we assume that $\max\big\{-n,\frac{p(n-\gamma)}{n}\big\}< \alpha <n(p-1)$, $1<p<\alpha+\gamma$ and $q=2$ unless otherwise stated. Let
$$
X:=\Big\{u\in W^{1}_p(\Omega_{\gamma},|x|^\alpha):\int_{\Omega_{\gamma}}u\,dx=0\Big\}.
$$

By Theorem~\ref{SPineq} and Theorem~\ref{eigen} we have the following Poincar\'e-Sobolev inequality
\[
\biggr(\int_{\Omega_{\gamma}}|u(y)-u_{\Omega_{\gamma}}|^2 dy\biggr)^{\frac{1}{2}} \leq B_{2,p}(\Omega_{\gamma},{|y|^\alpha})
\biggr(\int_{\Omega_{\gamma}}|\nabla u(y)|^p {|y|^\alpha} dy\biggr)^{\frac{1}{p}},\,\,\, {u}\in W^{1,p}(\Omega_{\gamma},{|y|^\alpha}),
\]
where $u_{\Omega_{\gamma}}=\frac{1}{|\Omega_{\gamma}|}\int_{\Omega_{\gamma}}u\,dy$.

Hence, we can endow the norm $\|\cdot\|_{X}$ on $X$ defined by
\begin{equation}\label{mfn}
\|u\|_{X}=\left(\int_{\Omega_{\gamma}}|\nabla u|^p \,|x|^\alpha\,dx\right)^\frac{1}{p}.
\end{equation}

The Lebesgue space $Y:=L_2(\Omega_{\gamma})$ be endowed with the norm
\begin{equation}\label{lqn}
\|u\|_{Y}:=\left(\int_{\Omega_{\gamma}}|u|^2\,dx\right)^\frac{1}{2}.
\end{equation}

\subsection{Existence results}

Let as formulate the regularity results for the Neumann $(p,q)$-eigenvalue problem.

\begin{theorem}\label{newthm}
Let $1<p<\gamma+\alpha$. Then the following properties hold:
\vskip 0.2cm
\noindent
$(a)$ There exists a sequence $\{\phi_n\}_{n\in\mathbb{N}}\subset X\cap Y$ such that $\|\phi_n\|_{Y}=1$ and for every $\psi\in X$, we have
\begin{equation}\label{its}
\int_{\Omega_{\gamma}}|\nabla\,\phi_{n+1}|^{p-2}\nabla\,\phi_{n+1} \nabla\,\psi\,|x|^\alpha\,dx=\mu_n\int_{\Omega_{\gamma}}\phi_{n}\psi\,dx,
\end{equation}
where
\begin{equation*}\label{subopmin}
\mu_n\geq \lambda:=\inf\left\{\int_{\Omega_{\gamma}}|\nabla\,\phi|^p\,|x|^\alpha\,dx:\phi\in X\cap {Y},\,\|\phi\|_{Y}=1\right\}.
\end{equation*}
\vskip 0.2cm
\noindent
(b) Moreover, the sequences $\{\mu_n\}_{n\in\mathbb{N}}$ and $\{\|\phi_{n+1}\|_{X}^{p}\}_{n\in\mathbb{N}}$ given by \eqref{its} are nonincreasing and converge to the same limit $\mu$, which is bounded below by $\lambda$. Further, there exists a subsequence $\{n_j\}_{j\in\mathbb{N}}$ such that both $\{\phi_{n_j}\}_{j\in\mathbb{N}}$ and $\{\phi_{n_{j+1}}\}_{j\in\mathbb{N}}$ converges in $X$ to the same limit $\phi\in X\cap Y$ with $\|\phi\|_{Y}=1$ and $(\mu,\phi)$ is an eigenpair of \eqref{pr-m}.
\end{theorem}

\begin{theorem}\label{subopthm1}
Let $1<p<\gamma+\alpha$ and $q=2$. Suppose $\{u_n\}_{n\in\mathbb{N}}\subset X\cap Y$ is a minimizing sequence for $\lambda$, that is $\|u_n\|_Y=1$ and $\|u_n\|_X^{p}\to\lambda$. Then there exists a subsequence $\{u_{n_j}\}_{j\in\mathbb{N}}$ which converges weakly in $X$ to $u\in X\cap Y$ such that $\|u\|_Y=1$ and
$
\lambda=\|u\|^p_{X}.
$
Moreover, $u$ is an eigenfunction of \eqref{pr-m} corresponding to $\lambda$ and its associated eigenfunctions are precisely the scalar multiple of those vectors at which $\lambda$ is reached.
\end{theorem}

\begin{remark}
Theorems~\ref{newthm}--\ref{subopthm1} are correct in the case of Lipschitz domains $\Omega\subset\mathbb R^n$. In this case $p^{*}_{n}=np/(n-p)$, where $1<p<n$.
\end{remark}

\subsection{Operators associated to eigenvalue problems}

Let us denote by $X^*$ and $Y^*$ denotes the dual of $X$ and $Y$ respectively.
Let us define the operator $A:X\to X^*$ by
\begin{equation}\label{ma}
\begin{split}
\langle A\phi,\psi\rangle&=\int_{\Omega_{\gamma}}|\nabla \phi|^{p-2}\nabla\phi\nabla\psi\,|x|^\alpha\,dx,\quad \forall \phi,\psi\in X
\end{split}
\end{equation}
and $B:Y\to Y^*$ by
\begin{equation}\label{mb}
\begin{split}
\langle B\phi,\psi\rangle&=\int_{\Omega_{\gamma}}\phi\psi\,dx,\quad \forall \phi,\psi\in Y.
\end{split}
\end{equation}
First we state some useful results. The following result from \cite[Theorem $9.14$]{var} will be useful for us.
\begin{theorem}\label{MBthm}
Let $V$ be a real separable reflexive Banach space and $V^{*}$ be the dual of $V$. Assume that $A:V\to V^{*}$ is a bounded, continuous, coercive and monotone operator. Then $A$ is surjective, i.e., given any $f\in V^{*}$, there exists $u\in V$ such that $A(u)=f$. If $A$ is strictly monotone, then $A$ is also injective.
\end{theorem}
Moreover, for the following algebraic inequality, see \cite[Lemma $2.1$]{Dama}.
\begin{lemma}\label{alg}
Let $1<p<\infty$. Then for any $a,b\in\mathbb{R}^N$, there exists a positive constant $C=C(p)$ such that
\begin{equation}\label{algineq}
(|a|^{p-2}a-|b|^{p-2}b, a-b)\geq
C(|a|+|b|)^{p-2}|a-b|^2.
\end{equation}
\end{lemma}

Next, we prove the following result.
\begin{lemma}\label{newlem}
$(i)$ The operators $A$ defined by \eqref{ma} and $B$ defined by \eqref{mb} are continuous. $(ii)$ Moreover, $A$ is bounded, coercive and monotone.
\end{lemma}
\begin{proof}
\noindent
$(i)$ \textbf{Continuity:} We only prove the continuity of $A$, since the continuity of $B$ would follow similarly. To this end, suppose $\phi_n\in X$ such that $\phi_n\to v$ in the norm of $X$. Thus, up to a subsequence $\nabla \phi_{n}\to \nabla v$ in $\Omega_{\gamma}$. We observe that
\begin{equation}\label{mfd}
\||x|^\frac{\alpha}{p'}|\nabla v_{n}|^{p-1}\|_{L^{p'}(\Omega_{\gamma})}\leq c,
\end{equation}
for some constant $c>0$, which is independent of $n$. Thus, up to a subsequence, we have
\begin{equation}\label{fc}
|x|^\frac{\alpha}{p'}|\nabla \phi_n|^{p-2}\nabla \phi_n{\rightharpoonup} |x|^\frac{\alpha}{p'}|\nabla \phi|^{p-1}\nabla \phi\text{ weakly in }L^{p'}(\Omega_{\gamma}).
\end{equation}
Since, the weak limit is independent of the choice of the subsequence, as a consequence of \eqref{fc}, we have
$$
\lim_{n\to\infty}\langle A\phi_n,\psi\rangle=\langle A\phi,\psi\rangle
$$
for every $\psi\in X$. Thus $A$ is continuous.
\vskip 0.2cm
\noindent
$(ii)$ \textbf{Boundedness:} Using the estimate \eqref{mest} obtained below, we have
$$
\|A\phi\|_{X^*}=\sup_{\|\psi\|_X\leq 1}|\langle A\phi,\psi\rangle|\leq\|\phi\|_X^{p-1}\|\psi\|_X\leq\|\phi\|^{p-1}_X.
$$
Thus, $A$ is bounded.

\noindent
\textbf{Coercivity:}  We observe that
$$
\langle A\phi,\phi\rangle=\|\phi\|_{X}^p.
$$
Since $p>1$, we have $A$ is coercive.

\noindent
\textbf{Monotonicity:} Using Lemma \ref{alg}, it follows that there exists a constant $C=C(p)>0$ such that
for every $\phi,\psi\in X$, we have
\begin{multline*}
\langle A\phi-A\psi,\phi-\psi\rangle=\int_{\Omega_{\gamma}}(|\nabla\,\phi|^{p-2}\nabla\,\phi-|\nabla\,\psi|^{p-2}\nabla\,\psi,\nabla\,(\phi-\psi))\,|x|^\alpha\,dx\\
=
\int_{\Omega_{\gamma}}(|\nabla\,\phi|^{p-2}\nabla\,\psi-|\nabla\,\psi|^{p-2}\nabla\,\psi,\nabla\,\phi-\nabla\,\psi)\,|x|^\alpha\,dx
\\
 \geq C(p)\int_{\Omega_{\gamma}}(|\nabla\,\phi|+|\nabla\,\psi|)^{p-2}|\nabla\,\phi-\nabla\,\psi|^2\,|x|^\alpha\,dx
\geq 0.
\end{multline*}
Thus, $A$ is a monotone operator.
\end{proof}

\begin{lemma}\label{auxlmab}
The operators $A$ defined by \eqref{ma} and $B$ defined by \eqref{mb} satisfy the following properties:
\vskip 0.2cm
\noindent
$(H_1)$ $A(t\phi)=|t|^{p-2}tA(\phi)\quad\forall t\in\R\quad \text{and}\quad\forall \phi\in X$.
\vskip 0.2cm

\noindent
$(H_2)$ $B(t\psi)=tB(\psi)\quad\forall t\in\R\quad \text{and}\quad\forall \psi\in Y$.
\vskip 0.2cm

\noindent
$(H_3)$ $\langle A(\phi),\psi\rangle\leq\|\phi\|_X^{p-1}\|\psi\|_X$ for all $\phi,\psi\in X$, where the equality holds if and only if $\phi=0$ or $\psi=0$ or $\phi=t \psi$ for some $t>0$.
\vskip 0.2cm

\noindent
$(H_4)$ $\langle B(\phi),\psi\rangle\leq\|\phi\|_Y^{p-1}\|\psi\|_Y$ for all $\phi,\psi\in Y$, where the equality holds if and only if $\phi=0$ or $\psi=0$ or $\phi=t\psi$ for some $t\geq 0$.
\vskip 0.2cm

\noindent
$(H_5)$ For every $\psi\in Y\setminus\{0\}$ there exists $u\in X\setminus\{0\}$ such that
$$
\langle Au,\phi\rangle=\langle B(\psi),\phi\rangle\quad\forall\quad \phi\in X.
$$
\end{lemma}
\begin{proof}
\vskip 0.2cm

\noindent
$(H_1)$ Follows by the definition of $A$ in \eqref{ma}.

\noindent
$(H_2)$  Follows by the definition of $B$ in \eqref{mb}.
\vskip 0.2cm

\noindent
$(H_3)$ First using Cauchy-Schwartz inequality and then by H\"older's inequality with exponents $p{'}=\frac{p}{p-1}$ and $p$, for every $v,w\in X$, we obtain
\begin{equation}\label{mest}
\begin{split}
\langle A\phi,\psi\rangle&=\int_{\Omega_{\gamma}}|\nabla \phi|^{p-2}\nabla \phi\nabla \psi\,|x|^\alpha\,dx\leq\int_{\Omega_{\gamma}}|\nabla \phi|^{p-1}|\nabla \psi|\,|x|^\alpha\,dx\\
&\leq\Big(\int_{\Omega_{\gamma}}|\nabla \phi|^p\,|x|^\alpha\,dx\Big)^\frac{p-1}{p}\Big(\int_{\Omega_{\gamma}}|\nabla \psi|^p\,|x|^\alpha\,dx\Big)^\frac{1}{p}=\|\phi\|_X^{p-1}\|\psi\|_X.
\end{split}
\end{equation}
Let the equality
\begin{equation}\label{mequal}
\langle A(\phi),\psi\rangle=\|\phi\|_X^{p-1}\|\psi\|_X
\end{equation}
holds for every $\phi,\psi\in X$. We claim that either $\phi=0$ or $\psi=0$ or $\phi=t\psi$ for some constant $t>0$. Indeed, if $\phi=0$ or $\psi=0$, this is trivial. Therefore, we assume $\phi\neq 0$ and $\psi\neq 0$ and prove that $\phi=t\psi$ for some constant $t>0$. By the estimate \eqref{mest} if the equality \eqref{mequal} holds, then we have
\begin{equation}\label{mequal1}
\langle A(\phi),\psi\rangle=\int_{\Omega_{\gamma}}|\nabla\phi|^{p-1}|\nabla\psi|\,|x|^\alpha\,dx,
\end{equation}
which gives  us
\begin{equation}\label{mCS}
\int_{\Omega_{\gamma}}f(x)\,dx=0,
\end{equation}
where
$$
f(x)=|\nabla \phi|^{p-1}|\nabla \psi|-|\nabla \phi|^{p-2}\nabla \phi\nabla \psi.
$$
By Cauchy-Schwartz inequality, we have $f\geq 0$ in $\Omega_{\gamma}$. Hence using this fact in \eqref{mCS}, we have $f=0$ in $\Omega_{\gamma}$, which reduces to
\begin{equation}\label{mfCS}
|\nabla\phi|^{p-1}\nabla\psi=|\nabla \phi|^{p-2}\nabla \phi\nabla \psi\text{ in }\Omega_{\gamma},
\end{equation}
which gives $\nabla{\phi}(x)=c(x)\nabla \psi(x)$ for some $c(x)\geq 0$.

On the otherhand, if the equality \eqref{mequal} holds, then by the estimate \eqref{mest} we have
\begin{equation}\label{mequal2}
\begin{split}
f_1=f_2,
\end{split}
\end{equation}
where
$$
f_1=\int_{\Omega_{\gamma}}|\nabla{\phi}|^{p-1}\nabla \psi\,|x|^\alpha\,dx,\quad
f_2=\left(\int_{\Omega_{\gamma}}|\nabla{\phi}|^p\,|x|^\alpha\,dx\right)^\frac{p-1}{p}\left(\int_{\Omega_{\gamma}}|\nabla\psi|^p\,|x|^\alpha\,dx\right)^\frac{1}{p}.
$$
Thus, we have $|\nabla{\phi}|(x)=d|\nabla\psi|(x)$ a.e. $x\in \Omega_{\gamma}$ for some constant $d>0$. Thus, we obtain $c(x)=d$. As a consequence, we get $\nabla{\phi}=d\nabla \psi$ a.e. in $\Omega_{\gamma}$ and therefore, we deduce that $\|\phi-d\psi\|_X=0$, which gives $\phi=d\psi$ a.e. in $\Omega_{\gamma}$. Hence, the property $(H3)$ is verified.
\vskip 0.2cm

\noindent
$(H_4)$ The hypothesis $(H_4)$ can be verified similarly.

\vskip 0.2cm

\noindent
$(H_5)$ We observe that $X$ is a separable and reflexive Banach space. By Lemma \ref{newlem}, the operator $A:X\to X^*$ is bounded, continuous, coercive and monotone.

By the Sobolev embedding theorem, we have $X$ is continuously embedded in $Y$. Therefore, $B(\psi)\in X^*$ for every $\psi\in Y\setminus\{0\}$.

Hence, by Theorem \ref{MBthm}, for every $\psi\in Y\setminus\{0\}$, there exists $u\in X\setminus\{0\}$ such that
$$
\langle A u,\phi\rangle=\langle B\psi,\phi\rangle\quad\forall \phi\in X.
$$
Hence the property $(H_5)$ holds. This completes the proof.
\end{proof}

\subsection{Proof of the regularity results:}
\vskip 0.2cm
\noindent
\textbf{Proof of Theorem \ref{newthm}:}
\vskip 0.2cm
\noindent
$(a)$ First we recall the definition of the operators $A:X\to X^*$ from \eqref{ma} and $B:Y\to Y^*$ from \eqref{mb} respectively. Then, noting the property $(H_5)$ from Lemma \ref{auxlmab} and proceeding along the lines of the proof in \cite[page $579$ and pages $584-585$]{Ercole}, the result follows.
\vskip 0.2cm
\noindent
$(b)$ We observe that $X$ is uniformly convex Banach space and by the Sobolev embedding theorem \cite{GG94,GU09}, $X$ is compactly embedded in $Y$. Next, using Lemma \ref{newlem}-$(i)$, the operators $A:X\to X^*$ and $B:Y\to Y^*$ are continuous and by Lemma \ref{auxlmab}, the properties $(H_1)-(H_5)$ holds. Noting these facts, the result follows from \cite[page $579$, Theorem 1]{Ercole}. \qed
\vskip 0.2cm
\noindent
\textbf{Proof of Theorem \ref{subopthm1}:} The proof follows due to the same reasoning as in the proof of Theorem \ref{newthm}-$(b)$ except that here we apply \cite[page $583$, Proposition $2$]{Ercole} in place of \cite[page $579$, Theorem 1]{Ercole}.
\vskip 0.2cm
\noindent

\textbf{Acknowledgements.}
The second author was supported by MSHER (agreement No. 075-02-2023-943).

\vskip 0.1cm

Department of Mathematical Sciences, Indian Institute of Science Education and Research Berhampur,
Berhampur, Odisha 760010, India

\emph{E-mail address:} \email{pgarain92@gmail.com} \\

Regional Scientific and Educational Mathematical Center, Tomsk State University, 634050 Tomsk, Lenin Ave. 36, Russia
							
\emph{E-mail address:} \email{va-pchelintsev@yandex.ru} \\
			
Department of Mathematics, Ben-Gurion University of the Negev, P.O.Box 653, Beer Sheva, 8410501, Israel
							
\emph{E-mail address:} \email{ukhlov@math.bgu.ac.il}


\begin{thebibliography}{99}

\bibitem{BCDL16} B.~Brandolini, F.~Chiacchio, E.~B.~Dryden, J.~J.~Langford, Sharp Poincar\'e inequalities in a class of non-convex sets, J. Spectr. Theory, 8 (2018), 1583--1615.

\bibitem{BCT9} B.~Brandolini, F.~Chiacchio, C.~Trombetti, Sharp estimates for eigenfunctions of a Neumann problem, Comm. Partial Differential Equations, 34 (2009), 1317--1337.

\bibitem{BCT15} B.~Brandolini, F.~Chiacchio, C.~Trombetti, Optimal lower bounds for eigenvalues of linear and nonlinear Neumann problems, Proc. of the Royal Soc. of Edinburgh, 145A (2015), 31--45.

\bibitem{CHP} C.~Gisella; H.Antoine; P.~Giovanni, Corrigendum to "An isoperimetric inequality for a nonlinear eigenvalue problem", [Ann. I. H. Poincar\'e -- AN 29 (1) (2012) 21--34].

\bibitem{B10} H.~Brezis, Functional Analysis, Sobolev Spaces and Partial Differential Equations, Springer New York, NY, 1969.

\bibitem{var} P.~G.~Ciarlet, Linear and nonlinear functional analysis with applications, Society for Industrial and
Applied Mathematics, Philadelphia, PA, 2013.

\bibitem{Dama}
L.~Damascelli, Comparison theorems for some quasilinear degenerate elliptic operators and applications
to symmetry and monotonicity results, Ann. Inst. H. Poincar\'e Anal. Non Lin\'eaire, 15(4) (1998), 493--516.

\bibitem{Ercole}
G.~Ercole, Solving an abstract nonlinear eigenvalue problem by the inverse iteration method, Bull. Braz.
Math. Soc., (N.S.), 49(3) (2018), 577--591.

\bibitem{Fe69} H.~Federer, Geometric measure theory, Sp\-rin\-ger Verlag, Berlin, 1969.

\bibitem{FPR}
R.~Filippucci, P.~Pucci, M.~Rigoli, Nonlinear weighted $p$-Laplacian elliptic inequalities with gradient
terms, Commun. Contemp. Math., 12 (2010), 501--535.

\bibitem{GG94}
V.~Gol'dshtein, L.~Gurov, Applications of change of variables operators for exact embedding theorems,
Integral Equ. Oper. Theory, {19} (1994), 1--24.

\bibitem{GU09}
V.~Gol'dshtein, A.~Ukhlov, Weighted Sobolev spaces and embedding theorems,
Trans. Amer. Math. Soc., {361} (2009), 3829--3850.

\bibitem{GU16} V.~Gol'dshtein, A.~Ukhlov, On the first Eigenvalues of Free Vibrating Membranes in Conformal Regular Domains,
Arch. Rational Mech. Anal., 221 (2016), 893--915.

\bibitem{GU17} V.~Gol'dshtein, A.~Ukhlov, The spectral estimates for the Neumann-Laplace operator in space domains, Adv. Math., 315 (2017), 166--193.

\bibitem{HKM} J.~Heinonen, T.~Kilpelinen, O.~Martio, Nonlinear Potential Theory of Degenerate Elliptic Equations. Clarendon Press. Oxford,
New York, Tokio. 1993.

\bibitem{Kufn} A.~Kufner, Weighted Sobolev spaces, Leipzig, Teubner-Texte zur Mathematik, 1980.

\bibitem{M} V.~Maz'ya, Sobolev spaces: with applications to elliptic partial differential equations, Springer, Berlin/Heidelberg, 2010.

\bibitem{PW} L.~E.~Payne, H.~F.~Weinberger, An optimal Poincar\'e inequality for convex domains, Arch. Rat. Mech. Anal., 5 (1960), 286-292.

\bibitem{PU23}
V.~Pchelintsev, A.~Ukhlov, On regularity of weighted Sobolev homeomorphisms. J. Pure and Applied Functional Analysis (in press).

\bibitem{PS51} G.~P\'olya, G.~Szeg\"o, Isoperimetric Inequalities in Mathematical Physics, Princeton University Press, 1951.

\bibitem{U93}
A.~Ukhlov, On mappings, which induce embeddings of Sobolev spaces, Siberian Math. J. {34} (1993), 185--192.

\bibitem{UV08}
A.~Ukhlov, S.~K.~Vodop'yanov, Mappings associated with weighted Sobolev spaces, Complex Analysis and Dynamical Systems III, Contemporary Mathematics Series, 455 (2008), 369--382.

\bibitem{VGR79} S.~K.~Vodop'yanov, V.~M.~Gol'dshtein, Yu.~G.~Reshetnyak, On geometric properties of functions with generalized first derivatives, Uspekhi Mat. Nauk, 34 (1979), 17--65. 

\bibitem{VU04} S.~K.~Vodop'yanov, A.~D.~Ukhlov, Set functions and their applications in the theory of Lebesgue and Sobolev spaces,
Siberian Adv. in Math, 14 (2004), 78--125.

\bibitem{VU05} S.~K.~Vodop'yanov, A.~D.~Ukhlov, Set functions and their applications in the theory of Lebesgue and Sobolev spaces,
Siberian Adv. in Math, 15 (2005), 91--125.


\end{thebibliography}
\end{document}